\documentclass{amsart}
\usepackage{enumitem}
\usepackage[a4paper,margin= 3 cm]{geometry}
\usepackage{tikz-cd}
\usepackage[hyphens]{url}
\usepackage{amsthm,amsmath,amssymb,amsfonts,dsfont}
\usepackage[colorlinks, linkcolor=blue, citecolor=blue]{hyperref}
\usepackage[indentafter]{titlesec}
\usepackage{mathrsfs}
\usepackage{bm}
\usepackage{titlesec}

\numberwithin{equation}{section}

\titleformat{name=\section}{}{\thetitle.}{0.8em}{\centering\scshape}

\titleformat{\subsection}
  {\normalfont\bfseries}
  {\thesubsection}
  {1em}
  {}

\theoremstyle{plain}
\newtheorem{theorem}{\textbf{Theorem}}[section]
\newtheorem{lemma}[theorem]{\textbf{Lemma}}
\newtheorem{proposition}[theorem]{\textbf{Proposition}}

\newtheorem{letterthm}{Theorem}

\theoremstyle{definition}
\newtheorem{definition}[theorem]{\textbf{Definition}}
\newtheorem{example}[theorem]{\textbf{Example}}
\newtheorem{remark}[theorem]{\textbf{Remark}}

\def\a{\alpha}
\def\b{\beta}

\title[On the K-theory of algebraic Cuntz-Pimsner rings]{On the K-theory of algebraic Cuntz-Pimsner rings}
\author{Thibaut Lescure}
\address{LMNO\\ 6 Boulevard Maréchal Juin\\ 14000 Caen \\ FRANCE}
\email{thibaut.lescure@etu.unicaen.fr}
\begin{document}

\begin{abstract}
We establish a long exact sequence for the homotopy K-theory groups of the algebraic Cuntz-Pimsner rings introduced by Carlsen and Ortega \cite{CO11} by adapting Pimsner's original proof \cite{Pim97} to Cuntz's formalism.
\end{abstract}
 \maketitle

\section{Introduction}

In \cite{Pim97}, Pimsner introduced a class of $C^*$-algebras generalizing both crossed products by $\mathbb{Z}$ and graph $C^*$-algebras. The construction takes as input a coefficient $C^*$-algebra $A$ and a right Hilbert module $\mathcal{H}$ together with a left action of $A$ on $\mathcal{H}$ by adjointable operators. By seeing the elements of $\mathcal{H}$ as creation and annihilation operators on the Fock space one defines a $C^*$-algebra called the Toeplitz(-Pimsner) algebra $\mathcal{T}_{\mathcal{H}}$. If $\mathcal{I}$ is a $C^*$-ideal of $A$ that acts on the left on $\mathcal{H}$ by compact operators, then quotienting by a suitable ideal yields the relative Cuntz-Pimsner algebra $\mathcal{O}_{\mathcal{H},\mathcal{I}}$. Using Kasparov's bivariant K-theory, Pimsner was able to build an inverse in K-theory to the inclusion map $A \rightarrow \mathcal{T}_{\mathcal{H}}$. This result together with a computation of the induced maps $K_*(\mathcal{I}) \rightarrow K_*(A)$ gives a long exact sequence of K-theory groups generalizing both that of Pimsner-Voiculescu and that for graph $C^*$ algebras.

Carlsen and Ortega \cite{CO11} introduced an analogue of Pimsner's constructions to the purely algebraic setting in which $A$ is replaced by a (non-unital, non-commutative) ring $R$ and $\mathcal{H}$ is replaced by an $R$-system $\mathcal{X} = (X,X',g)$ which consists of a pair of $R$-bimodules and an $R$-bimodule map $X' \otimes_R X \rightarrow R$. 
In \cite{CT06}, Corti\~nas and Thom introduced an algebraic analogue of Kasparov's bivariant K-theory carrying the classical approach of Cuntz  \cite{CMR07,C97} to a purely discrete setting. Their theory is related to Weibel's homotopy K-theory \cite{Wei89} in the same way as operator K-theory is related to Kasparov's KK-theory. In the same paper, they established an algebraic analogue of the Pimsner–Voiculescu exact sequence for the homotopy K-theory groups of a crossed product of the form $R \rtimes \mathbb{Z}$. In \cite{ABC09} a long exact sequence of homotopy K-theory groups was established for Leavitt path algebras of a row-finite quiver over an arbitrary ring. These two classes of rings are particular cases of the so-called algebraic Cuntz-Pimsner rings.

We rewrite Pimsner's original proof using a suitable quasi-homomorphism as an inverse in $K$-theory to the inclusion $R \rightarrow \mathcal{T}_\mathcal{X}$. Instead of working directly with homotopy $K$-theory or its bivariant version, we work with a given homotopy invariant, split-exact and $M$-stable functor $E : \mathbf{Rings} \rightarrow \mathbf{Ab}$ (see \cite{CMR07} or \cite{CT06}). We need to ensure that the underlying module structures are sufficiently non-degenerate. In particular the $R$-system will be assumed to satisfy condition (FS) of \cite{CO11} to ensure that the Toeplitz ring is universal. Moreover, for all $R$-system $\mathcal{X}$ and any $M$-stable functor $E$, if $\mathcal{K}_R(\mathcal{X})$ denotes the ring of finite-rank operators then there should be an induced map $E(\mathcal{K}_R(\mathcal{X})) \rightarrow E(R)$. Corner embeddings should be sent to isomorphisms by $E$.

For these reasons, we assume that $R$ has local units, that the $R$-system $\mathcal{X}$ satisfies condition (FS) and comes equipped with a functional homomorphism to a canonical $R$-system $R^{(I)}$ (see Definition 2.12). Under these assumptions we are able to solve these technical problems using the formalism introduced by Burgstaller \cite{Bur25}. We then recall the construction of the Toeplitz and relative Cuntz-Pimsner rings and proceed with the main computation. Our first result is the following. \begin{letterthm}
    Let $R$ be a ring with local units. Let $\mathcal{X}$ be an $R$-correspondence. Let $\mathcal{T}_{\mathcal{X}}$ be the corresponding Toeplitz ring.
    Every homotopy invariant, split-exact and $M$-stable functor $E$ sends the inclusion $R \rightarrow \mathcal{T}_{\mathcal{X}}$ to an isomorphism $E(R) \cong E(\mathcal{T}_{\mathcal{X}})$.
\end{letterthm}
Let $(E_n)_{n \in \mathbb{Z}} : \mathbf{Rings} \rightarrow \mathbf{Ab}$ be a sequence of homotopy invariant and $M$-stable functors which satisfies excision. This means that any extension of rings \[\begin{tikzcd}
I \arrow[r] & R \arrow[r] & R/I
\end{tikzcd} \]
induces a long exact sequence of the corresponding abelian groups \[\begin{tikzcd}
\ldots \arrow[r] & E_n(I) \arrow[r] & E_n(R) \arrow[r] & E_n(R/I) \arrow[r] & E_{n-1}(I) \arrow[r] & \ldots
\end{tikzcd}\]
Using Theorem A and the fact that the Cuntz-Pimsner ring is a quotient of the Toeplitz ring we obtain \begin{letterthm}
  Let $R$ be a ring with local units. Let $\mathcal{X}$ be an $R$-correspondence. Let $\mathcal{I} \lhd R$ be a two-sided ideal that has local units and acts on $X$ on the left by compact operators. There is a long exact sequence $(n\in \mathbb{Z})$
\[\begin{tikzcd}
\ldots \arrow[r] & E_{n}(\mathcal{I}) \arrow[rr, "E_n(i) - E_n(\mathcal{X})"] &  & E_{n}(R) \arrow[r, "E_n(j)"] & {{ E_{n}(\mathcal{O}_{\mathcal{X},\mathcal{I}}})} \arrow[r] & E_{n-1}(\mathcal{I}) \arrow[r] & \ldots
\end{tikzcd}\]
Here $j : R \rightarrow \mathcal{O}_{\mathcal{X}, \mathcal{I}}$ and $i : \mathcal{I} \rightarrow R$ are the natural inclusions.
\end{letterthm}
Theorems A and B are both true if we replace the sequence of functors $(E_n)_{n \in \mathbb{Z}}$ by either Weibel's homotopy K-theory \cite{Wei89} or by periodic cyclic homology \cite{CQ97}.

\section{Preliminaries}\label{Preliminaries}

\subsection{Burgstaller's functional modules and $M$-stable functors}

We fix a (not necessarily unital or commutative) ring $R$.
\begin{definition}
   A right functional $R$-module is a triple $(X,X',g)$ where $X$ is a right $R$-module, $X'$ is a left $R$-module and $g : X'\otimes_{\mathbb{Z}}X \rightarrow R$ is an $R$-bimodule map (often called the scalar product of the functional module).
    For all $f\in \text{Hom}_{R}(X,X)$ we say that $f$ is adjointable with adjoint $f^{*}\in \text{Hom}_{R}(X',X')$ if
    \[\forall x \in X, \forall \phi \in X', g(\phi \otimes f(x)) = g(f^{\ast}(\phi) \otimes x)\]
        We call $g$ non-degenerate if
    \[(\forall \phi \in X', g(\phi \otimes x) = 0) \implies x = 0 \text{ and }(\forall x \in X, g(\phi \otimes x) = 0) \implies \phi = 0\]
\end{definition}
We will write $g(\phi \otimes x) = \phi(x)$ for $x \in X$ and $\phi$ in $X'$. One easily defines the direct sum of two right functional $R$-modules $\mathcal{X} = (X,X',g)$ and $\mathcal{Y} = (Y,Y',h)$ as $\mathcal{X}\oplus \mathcal{Y} = (X\oplus Y, X'\oplus Y',g +h)$. The following proposition is easy to verify.

\begin{proposition}
    For each right functional module $\mathcal{X} = (X,X',g)$ over $R$, the set of adjointable endomorphisms  of $\mathcal{X}$ forms a ring.
    If $g$ is non-degenerate then the adjoint is unique and the adjoints satisfy $(f_{1}f_{2})^{*} = f_{2}^{*}f_{1}^{*}$.
\end{proposition}
We denote by $\mathcal{L}_{R}(\mathcal{X})$ the ring of adjointable right $R$-module endomorphisms on the functional module $\mathcal{X}$, with pointwise sum and composition as operations. We will always be working with functional modules with non-degenerate scalar product. Hence for two functional $R$-modules $\mathcal{X}$ and $\mathcal{Y}$ we will often write $f : \mathcal{X} \rightarrow \mathcal{Y}$ for an $R$-module map $f : X \rightarrow Y$ admitting an adjoint.
\begin{definition}
    Let $\mathcal{X} = (X,X',g)$ be a right functional module over $R$. We define the abelian group of compact operators of $\mathcal{X}$ to be $\mathcal{K}_{R}(\mathcal{X}) = X\otimes_{R} X'$.
    It is a ring with the product defined by $(x_{1}\otimes \phi_{1})(x_{2}\otimes \phi_{2}) = x_{1}\otimes (\phi_{1}(x_{2})\cdot\phi_{2})$.
\end{definition}
$\mathcal{K}_{R}(\mathcal{X})$ acts on the left on $X$ by $(x\otimes \phi)\cdot y = x\cdot \phi(y)$ but also on $X'$ on the right by $\psi \cdot(x\otimes \phi) = \psi(x) \cdot \phi$. This gives a ring homomorphism from $\mathcal{K}_{R}(\mathcal{X})$ into the adjointable operators of $\mathcal{X}$.
\begin{definition}{\cite{CO11}}
    A right functional $R$-module $\mathcal{X} = (X,X',g)$ is said to satisfy condition (FS) if for all $x_1,\ldots x_n \in X$ and $\phi_1, \ldots \phi_n \in X'$ there exist $\Theta_1, \Theta_2 \in \mathcal{K}_R(\mathcal{X})$ such that $\Theta_1 \cdot x_i = x_i$ and $\phi_i \cdot \Theta_2 = \phi_i$ for all $1 \leq i \leq n$.
\end{definition}
Recall that a ring $R$ is said to have local units if for any $r_1, \ldots, r_n \in R$ there exists an idempotent $e\in R$ such that for all $1\leq i \leq n$ we have $er_i = r_i e = r_i$.
Recall that a right $R$-module $M$ is called non-degenerate if $M\cdot R = M$.  If $R$ has local units then non-degeneracy is equivalent to asking that the right action map $M \otimes_R R \rightarrow M$ is bijective.
\begin{proposition}{\cite{CO11}} Let $\mathcal{X}$ be a right functional module over $R$ satisfying condition (FS). \begin{enumerate}[label=(\roman*)]
    \item The underlying scalar product of $\mathcal{X}$ is non-degenerate.
    \item The underlying $R$-module structures of $\mathcal{X}$ are non-degenerate.
    \item $\mathcal{K}_{R}(\mathcal{X})$ is embedded in $ \mathcal{L}_{R}(\mathcal{X})$ as a two-sided ideal by the ring homomorphism
$j : x\otimes \phi \mapsto \theta_{x,\phi}$ defined by $\theta_{x,\phi}(y) = x\cdot \phi(y)$.
\end{enumerate}
\end{proposition}
\begin{proof}
Let $x\in X$ be such that for all $\phi \in X'$, $\phi(x) = 0$. There exists a $\Theta \in \mathcal{K}_{R}(\mathcal{X})$ such that $\Theta \cdot x = x$. We can write $\Theta = \sum_{i} x_{i}\otimes \phi_{i}$ thus $x = \sum_{i} x_{i} \cdot \phi_{i}(x) = 0$. Similarly if $\phi \in X'$ is such that $ \forall x \in X, \phi(x) = 0$ then $\phi = 0$.

The formulas $f \circ \theta_{x, \phi} = \theta_{f(x), \phi}$ and $\theta_{x, \phi} \circ f = \theta_{x, f^{*}(\phi)}$ ensure that the image of $ \mathcal{K}_{R}(\mathcal{X})$ is a two-sided ideal of $\mathcal{L}_{R}(\mathcal{X})$.
Let $k\in \mathcal{K}_{R}
(\mathcal{X})$ be such that $j(k) = 0$, i.e. $k\cdot x = 0$ for all $x\in X$. Write $k = \sum_{i}x_i\otimes\phi_i$. Let $e \in \mathcal{K}_R(\mathcal{X})$ be such that $\phi_i \cdot e = \phi_i$ for all $1 \leq i \leq n$. We have $ke = \sum_{i} x_i\otimes(\phi_i\cdot e) = k$. Write now $e = \sum_{j}y_j\otimes\psi_j$. Finally, $k = ke = \sum_{j} (k\cdot y_j)\otimes\psi_j = 0$.
\end{proof}

In the rest of this article we will always be working with functional modules satisfying condition (FS). Hence we will identify the element $x\otimes \phi \in X\otimes_R X'$ with $\theta_{x,\phi} \in \mathcal{L}_{R}(\mathcal{X})$.

\begin{example}
 For any set $I$ we will denote by $R^{(I)}$ the $R$-bimodule of finitely supported sequences $(r_i)_{i \in I}$ of elements of $R$. We will also denote by $R^{(I)}$ the functional module $(R^{(I)},R^{(I)},\langle,\rangle_R)$ equipped with the scalar product given by $\langle(\a_i),(\b_i)\rangle = \sum_i\a_i\b_i$. We will make use of the identification $\mathcal{K}_R(R^{(I)}) \cong M_I(R)$ with the ring of matrices with values in $R$ having only a finite number of non-zero values. When $R$ has local units this functional module satisfies condition (FS).
\end{example}
\begin{definition}{\cite{Bur25}}
    Let $\mathcal{X} = (X,X',g)$ and $\mathcal{Y} = (Y,Y',h)$ be two right functional $R$-modules. A functional homomorphism is the data of two $R$-module maps $U : X \rightarrow Y$ and $V : X' \rightarrow Y'$ such that \[\forall \phi \in X', x\in X, V(\phi)(U(x)) = \phi(x)\]
\end{definition}
If the scalar product of $\mathcal{Y}$ is non-degenerate then the maps $U$ and $V$ are injective. (If $U(x) = 0$ then $\phi(x) = 0$ for all $\phi \in X'$).
\begin{example}
    Assume that we have three right functional $R$-modules such that $\mathcal{Y} = \mathcal{X} \oplus \mathcal{P}$. In this case we let $U : X \rightarrow Y$ and $V : X' \rightarrow Y'$ be the inclusions associated to the splittings $Y = X \oplus P$ and $Y' = X' \oplus P'$. In other words if $\iota :\mathcal{X} \rightarrow \mathcal{Y}$ and $\pi : \mathcal{Y} \rightarrow \mathcal{X}$ are the obvious inclusion and projection, we let $U = \iota$ and $V = \pi^*$. We have $V(\phi)(U(x)) = (\pi^*(\phi))(\iota(x)) = \phi(\pi(\iota(x))) = \phi(x)$. Let $\phi_1, \ldots \phi_n \in X'$ and $y \in Y$. $V(\phi_i)(y) = (\pi^*(\phi_i))(y) = \phi_i(\pi(y))$. Hence $(\iota,\pi^*)$ is a functional homomorphism.
\end{example}
\begin{example}
   Let $M$ be a finitely generated right $R$-module. Then $M$ is projective if and only if it is a direct summand of a right $R$-module of the form $(eR)^n$ for an idempotent $e\in R$ (see \cite{Abr83}). There is a right functional $R$-module $\mathcal{M} = (M,M^*,g)$, where $M^{*}= \text{Hom}_{R}(M,R)$ and $g$ is the evaluation map. We have $((eR)^n)^*\cong (Re)^n$. By a slight abuse of notation we denote by $(eR)^n$ the right functional $R$-module $((eR)^n,(Re)^n,\langle,\rangle)$. There exists a right functional $R$-module $\mathcal{N}$ such that 
   $\mathcal{M} \oplus \mathcal{N} \cong (eR)^n$. Moreover there is an obvious functional homomorphism $(eR)^n \rightarrow R^n$ given by the inclusions. Hence there is a functional homomorphism $\mathcal{M} \rightarrow R^n$.
\end{example}
    Let $(U,V) :\mathcal{X} = (X,X',g) \rightarrow \mathcal{Y} = (Y,Y',h)$ be a functional homomorphism between two right functional $R$-modules. The formula of Definition 2.7 makes the map $ \iota_{\mathcal{X},\mathcal{Y}} = U \otimes V : \mathcal{K}_R(\mathcal{X}) \rightarrow \mathcal{K}_R(\mathcal{Y})$ a ring homomorphism. Moreover one can check that if $\mathcal{X}$ satisfies condition (FS) then $\iota_{\mathcal{X},\mathcal{Y}}$ is injective.
Let $\mathbf{Rings}$ be the category of associative rings with ring homomorphisms. \begin{definition}
    Let $E : \mathbf{Rings} \rightarrow \mathbf{Ab}$ be a functor. $E$ is called $M$-stable if for any ring $S$, any set $I$ and any distinguished element $i \in I$ the diagonal embedding $\sigma_i : S \rightarrow M_I(S)$ at the $i$-th coordinate is sent to an isomorphism.
\end{definition}
According to \cite[Proposition~3.16]{CMR07} any such functor is invariant under
inner isomorphisms, and thus $E(\sigma_i)$ does not depend on the choice of $i$.

Burgstaller \cite{Bur25} recently extended this isomorphism to more general corner
embeddings by using functional homomorphisms. For the reader’s convenience, we recall his result together with its
proof.

\begin{proposition}
    Let $E : \mathbf{Rings} \rightarrow \mathbf{Ab}$ be an $M$-stable functor. Let $R$ be a ring. Let $\mathcal{X}$ be a right functional $R$-module admitting a functional homomorphism to a functional module of the form $R^{(I)}$, $I$ being a set. The map induced by the upper left corner embedding $E(\iota_{R,R\oplus\mathcal{X}}) : E(R) \rightarrow E(\mathcal{K}_R(R \oplus \mathcal{X}))$ is an isomorphism. 
\end{proposition}
\begin{proof}
    Write $D_\mathcal{X} = \mathcal{K}_{R}(R\oplus \mathcal{X})$. There is a functional homomorphism $(U,V) :\mathcal{X} \rightarrow R^{(I)}$. Let $a$ be a formal symbol not in $I$, write simply $I\cup a$ for $I\cup \{a\}$. $(U,V)$ gives another functional homomorphism $(\mathrm{id}_R \oplus U, \mathrm{id}_R\oplus V) :R \oplus \mathcal{X} \hookrightarrow R^{(I\cup a)}$ which induces a ring homomorphism $\rho = \iota_{\mathcal{X},R^{(I\cup a)}} : D_\mathcal{X} \rightarrow M_{I\cup\{a\}}(R)$. Denote by $e : R \hookrightarrow D_{\mathcal{X}}$ the upper left corner embedding induced by the split embedding of functional modules $R \hookrightarrow R \oplus \mathcal{X}$.
    \newline
    Consider the following diagram
\begin{center}
\begin{tikzcd}
                                                                                &                                                                         &  &                                                                    &  & M_{2}(D_\mathcal{X}) \arrow[d, "\phi"]    \\
D_\mathcal{X} \arrow[r, "\rho"'] \arrow[rrrrru, "j_1"] \arrow[rrrrru, "\cong"'] & M_{I\cup{a}}(R) \arrow[rr, "M_{I\cup a}(e)"']                           &  & M_{I\cup{a}}(D_\mathcal{X}) \arrow[rr, "j_2"'] \arrow[rr, "\cong"] &  & M_{I\cup{a}\cup I'\cup a'}(D_\mathcal{X}) \\
                                                                                & R \arrow[rr, "e"] \arrow[u, "i_a"] \arrow[lu, "e"'] \arrow[u, "\cong"'] &  & D_\mathcal{X} \arrow[u, "i_a'"'] \arrow[u, "\cong"]                &  &                                          
\end{tikzcd}
\end{center}
Here $e = \iota_{R, R \oplus \mathcal{X}}$ denotes the corner embedding $R \rightarrow D_\mathcal{X}$. $i_a$ and $i_a'$ denote the embeddings into the corner of the coordinate labelled $a$ in the corresponding matrix ring. $I'\cup a'$ is just a copy of $I\cup a$ and $j_2$ embeds $M_{I\cup{a}}(D_\mathcal{X})$ via identification with the coordinates of $I'\cup a'$. $j_1$ is the lower right corner embedding. We have the following identifications $M_{2}(D_\mathcal{X}) \cong \mathcal{K}_R((R \oplus\mathcal{X})^2))$ and $M_{I\cup{a}\cup I'\cup a'}(D_\mathcal{X}) \cong \mathcal{K}_R((R \oplus\mathcal{X})^{I\cup{a}\cup I'\cup a'})$. The map $\phi : M_2(D_\mathcal{X}) \rightarrow M_{I\cup{a}\cup I'\cup a'}(D_\mathcal{X})$ in the diagram above is associated to the following functional homomorphisms
\begin{center}

\begin{tikzcd}
(R\oplus \mathcal{X})^2 \arrow[rr, "\mathrm{id}_{R\oplus \mathcal{X}}\oplus \mathrm{id}_{R}\oplus f", hook] &  & R\oplus \mathcal{X} \oplus R^{(I\cup a)} \arrow[rr, "\alpha\oplus\beta", hook] &  & (R\oplus \mathcal{X})^{(I\cup a)} \oplus (R\oplus \mathcal{X})^{(I'\cup a')}
\end{tikzcd}
\end{center}
Here $f = (U,V) :\mathcal{X} \rightarrow R^{(I)}$ and
$\alpha : R\oplus \mathcal{X} \hookrightarrow(R\oplus \mathcal{X})^{(I\cup a)}$ is the embedding at the $a$-th coordinate. $\beta : R^{(I \cup a)} \hookrightarrow(R\oplus \mathcal{X})^{(I'\cup a')}$ is the embedding $R \hookrightarrow R\oplus \mathcal{X}$ at each coordinate. In the above diagram the isomorphisms are to be understood at the level of the corresponding abelian groups: by $M$-stability the maps $E(i_a)$, $E(i_a')$, $E(j_1)$, $E(j_2)$ are bijective.

A direct computation proves that this diagram commutes ($\beta$ induces the ring homomorphism $M_{I\cup a}(e)$). It remains to show that $E(e)$ is an isomorphism with inverse $E(i_a)^{-1}E(\rho)$. We have $E(i_a)^{-1}E(\rho)E(e) = \mathrm{id}_{E(R)}$ because $i_a = \rho e$.
\begin{multline*}
E(e)E(i_a)^{-1}E(\rho)
= E(i_a')^{-1}E(M_{I\cup \{a\}}(e))E(\rho) \\
= E(i_a')^{-1}E(j_2)^{-1}E(j_2)E(M_{I\cup \{a\}}(e))E(\rho) \\
= E(i_a')^{-1}E(j_2)^{-1}E(\phi)E(j_1)
\end{multline*}
The trick is now to write $E(j_2) = E(j_2')$ and $E(j_1) = E(j_1')$ where $j_1' : D_\mathcal{X} \hookrightarrow M_2(D_\mathcal{X})$ is the upper left corner embedding and $j_2' : M_{I\cup a}(D_\mathcal{X}) \hookrightarrow M_{I\cup a \cup I' \cup a}(D_\mathcal{X})$ is the embedding obtained by identifying the coordinates $I\cup a$ in $I\cup a \cup I' \cup a'$. As $\phi j_1' = j_2'i_a'$ we finally get $E(i_a')^{-1}E(j_2)^{-1} E(\phi) E(j_1) = E(i_a')^{-1}E(j_2')^{-1} E(\phi) E(j_1') = \mathrm{id}_{E(D_\mathcal{X})}$.
\end{proof}
\begin{definition}
    Let $R$ and $S$ be two rings.
    An $R$-$S$ correspondence is the data of a quadruple $(\mathcal{X},\Delta,\mathcal{U},I)$ where  \begin{enumerate}[label=(\roman*)]
        \item $\mathcal{X}$ is a right functional $S$-module satisfying condition (FS)
        \item $\Delta : R \rightarrow \mathcal{L}_S(\mathcal{X})$ is a ring homomorphism (called the left action) such that the induced left $R$-module structure on $X$ and right $R$-module structure on $X'$ are non-degenerate: $\Delta(R)\cdot X = X$ and $X' \cdot \Delta(R) = X'$
        \item $I$ is a set and $ \mathcal{U} : \mathcal{X} \rightarrow S^{(I)}$ is a functional homomorphism
    \end{enumerate}
    An $R$-correspondence is an $R$-$R$ correspondence.
\end{definition}
The notion of correspondence serves here as the minimal algebraic analogue of the classical notion of a $C^*$-correspondence.

We will often simply write $\mathcal{X}$ for the correspondence $(\mathcal{X},\Delta,\mathcal{U},I)$. Choose an $R$-$S$ correspondence $\mathcal{X} = (X, X',g)$ with functional homomorphism $ \mathcal{U}_1 : \mathcal{X} \rightarrow S^{(I)}$ and an $S$-$T$ correspondence $\mathcal{Y}= (Y, Y',h)$ with functional homomorphism
$\mathcal{U}_2 : \mathcal{Y} \rightarrow T^{(J)}$. One defines the tensor product $\mathcal{X}\otimes_S\mathcal{Y}$ as the $R$-$T$ correspondence given by $\mathcal{X}\otimes_S\mathcal{Y} = (X\otimes_S Y,Y'\otimes_S X',k)$ with $k$ defined by the formula $(\psi \otimes \phi)(x\otimes y) = \psi(\phi(x)\cdot y)$ for all $x\in X, y\in Y, \phi \in X', \psi \in Y'$. The functional homomorphism underlying $\mathcal{X}\otimes_S\mathcal{Y}$ is given by \[\begin{tikzcd}
X\otimes_S Y \arrow[r, "U_1 \otimes_S \mathrm{id}_Y"]     & S^{(I)}\otimes_S Y \arrow[r]  & Y^{(I)} \arrow[r, "\oplus_I U_2"]  & R^{(I\times J)} \\
Y'\otimes_S X' \arrow[r, "\mathrm{id}_{Y'}\otimes_S V_1"] & Y'\otimes_S S^{(I)} \arrow[r] & Y'^{(I)} \arrow[r, "\oplus_I V_2"] & R^{(I\times J)}
\end{tikzcd} \]

Let $\mathcal{X}$ be an $R$-$S$ correspondence. Denote by $\Delta : R \rightarrow \mathcal{L}_S(\mathcal{X})$ the left action. Assume that $R$ is such that $\forall r\in R, \Delta(r) \in \mathcal{K}_S(\mathcal{X})$. Let $E : \mathbf{Rings} \rightarrow \mathbf{Ab}$ be an $M$-stable functor.
As $E(\iota_{R,R\oplus\mathcal{X}})$ is an isomorphism we see that $\mathcal{X}$ induces a map $E(\mathcal{X}) : E(R) \rightarrow E(S)$ by the following diagram \[\begin{tikzcd}
R \arrow[r, "\Delta"] & \mathcal{K}_S(\mathcal{X}) \arrow[r, "{\iota_{\mathcal{X},S\oplus\mathcal{X}}}"] & \mathcal{K}_S(S\oplus\mathcal{X}) & S \arrow[l, "{\iota_{S,S\oplus\mathcal{X}}}"']
\end{tikzcd}\]

For the applications, we will need another way of computing $E(\mathcal{X})$. The functional homomorphism of $\mathcal{X}$ yields the existence of a ring homomorphism $\iota_{\mathcal{X}, S^{(I)}} :\mathcal{K}_{S}(\mathcal{X}) \rightarrow M_I(S)$. As $E$ is an $M$-stable functor there is a map $E(\mathcal{K}_{S}(\mathcal{X})) \rightarrow E(S)$ obtained by composing with the inverse of the map induced by any corner embedding $\iota_{S,S^{(I)}} : S \rightarrow M_I(S)$.
\begin{proposition}
    $E(\iota_{S,S\oplus\mathcal{X}})^{-1}E(\iota_{\mathcal{X},S\oplus \mathcal{X}}) = E(\iota_{S,S^{(I)}})^{-1} E(\iota_{\mathcal{X},S^{(I)}})$.
\end{proposition}
\begin{proof}
The following diagram \begin{center}
    \begin{tikzcd}
\mathcal{K}_S(\mathcal{X}) \arrow[r, "{\iota_{\mathcal{X},S\oplus\mathcal{X}}}"] \arrow[rd, "h"'] & \mathcal{K}_S (S \oplus\mathcal{X}) \arrow[d, "\rho"] & S \arrow[l, "{\iota_{S,S\oplus\mathcal{X}}}"'] \arrow[ld, "{\iota_{a}}"] \\
                                                                                                  & M_{I\cup a}(S)                  &                                                                                       
\end{tikzcd}
\end{center}
is commutative. $\rho = \iota_{S\oplus \mathcal{X},S^{(I\cup a)}}$ is induced by the splitting $S^{(I \cup a)} = (S\oplus \mathcal{X}) \oplus \mathcal{Y}$. $h = \iota_{\mathcal{X},S^{(I\cup a)}}$ is the composition of $\iota_{\mathcal{X}, S^{(I)}} : \mathcal{K}_{S}(\mathcal{X}) \rightarrow M_I(S)$ with the embedding $\sigma = \iota_{S^{(I)},S^{(I\cup a)} } :M_I(S) \hookrightarrow M_{I\cup a}(S)$. $\iota_a = \iota_{S,S^{(I\cup a)}}$ is the embedding at the coordinate $a$. Its commutativity yields $E(\iota_{S,S\oplus\mathcal{X}})^{-1} E(\iota_{\mathcal{X},S\oplus\mathcal{X}}) = E(i_a)^{-1}E(h)$. The diagram
\begin{center}
    \begin{tikzcd}
\mathcal{K}_S(\mathcal{X}) \arrow[r, "{\iota_{\mathcal{X},S^{\small(I)}}}"] \arrow[rd, "h"'] & M_I(S) \arrow[d, "\sigma"] & S \arrow[l, "{\iota_{S,S^{\small(I)}}}"'] \arrow[ld, "i_a"] \\
                                                                                             & M_{I\cup a}(S)             &                                                            
\end{tikzcd}
\end{center}
is also commutative. Similarly, we get $E(\iota_{S,S^{(I)}})^{-1} E(\iota_{\mathcal{X},S^{(I)}}) = E(i_a)^{-1}E(h)$.
\end{proof}

\subsection{Algebraic Cuntz-Pimsner rings}\label{CP}

In this section $R$ is a ring with local units and $\mathcal{I}$ is a two-sided ideal of $R$ which also admits local units. Let $\mathcal{X} = (X,X',g)$ be an $R$-correspondence such that $\mathcal{I}$ acts on $X$ on the left by compact operators. This means that if $\Delta : R \longrightarrow{} \mathcal{L}_{R}(X)$ is the left action then 
$\Delta(\mathcal{I}) \subset \mathcal{K}_{R}(X)$.

Let
$T(X) = \bigoplus_{n \in \mathbb{N}}{X^{\otimes n}}$
be the $\mathbb{N}$-graded $R$-$R$ bimodule known as the Fock space (where the tensor products are taken over $R$).
\newline
There is an $R$-correspondence
\[
T(\mathcal X)=(T(X),T(X'),\widetilde g)
\]
where $\widetilde g(\phi_{1}\otimes\cdots\otimes\phi_{n},
x_{1}\otimes\cdots\otimes x_{m}) = 0$ if $n\neq m$, and 
\[
\widetilde g(\phi_{1}\otimes\cdots\otimes\phi_{n},
x_{1}\otimes\cdots\otimes x_{n})
= \phi_{1}\bigl(\phi_{2}(\cdots(\phi_{n}(x_{1})x_{2})\cdots)x_{n}\bigr).
\]
for all $\phi_k \in X'$ and $x_i \in X$. For $\phi\in X'$, $x\in X$ and $p \in T(X)$ a pure tensor, define
\begin{align*}
T_x(p) &:= x\otimes p,\\
T_\phi(p) &:= 
\begin{cases}
0 & \text{if }\deg(p)=0\\
\phi(p_1)\cdot p_2\otimes\cdots\otimes p_n
& \text{if }p=p_1\otimes\cdots\otimes p_n,\ n\ge1.
\end{cases}
\end{align*}

Then $T_x,T_\phi\in\mathcal L_R(T(X))$ with adjoints given on all pure tensors $\psi \in T(X')$ by
\begin{align*}
T_x^*(\psi) &:= 
\begin{cases}
0 & \text{if }\deg(\psi)=0,\\
\psi_1\otimes\cdots\otimes\psi_{n-1}\cdot\psi_n(x),
& \text{if }\psi=\psi_1\otimes\cdots\otimes\psi_n
\end{cases}\\
T_\phi^*(\psi) &:= \psi\otimes\phi,
\end{align*}
in $\mathcal L_R(T(X'))$.

\begin{definition}
    The \textbf{Toeplitz ring} 
    $\mathcal{T}_{\mathcal{X}}$ of $\mathcal{X}$ is the subring generated by the $T_{x}$, $T_{\phi}, r\cdot $id$_{T(X)} \in \mathcal{L}_{R}(T(X)), \phi\in X', x \in X, r\in R$.
    \newline
    Let $J_{\mathcal{X},{\mathcal{I}}}$ be the two-sided ideal of $\mathcal{T}_{\mathcal{X}}$ generated by $\mathcal{I}\cdot P_{0}$ where $P_{0}\in \mathcal{L}_{R}(T(X))$ is the map that sends $x$ to 0 if deg$(x) \geq 1$ and is the identity on $R$.
    The \textbf{Cuntz-Pimsner ring} of $\mathcal{X}$ with respect to the ideal $\mathcal{I}$  is the quotient ring $
\mathcal O_{\mathcal X,\mathcal I}
= \mathcal T_{\mathcal X} / J_{\mathcal X,\mathcal I}.$
\end{definition}
 $T(X)$ is an $\mathbb{N}$-graded $R$-bimodule. The rings $\mathcal{T}_\mathcal{X}$ and $\mathcal{O}_{\mathcal{X},\mathcal{I}}$ are $\mathbb{Z}$-graded. Observe that these rings have local units if the base ring $R$ has local units.
\begin{remark}
    Be careful that $P_{0}$ is not in general an element of $\mathcal{T}_{\mathcal{X}}$. The reason why $\mathcal{I}\cdot P_{0}$ is included in $\mathcal{T}_{\mathcal{X}}$ is because $\mathcal{I}$ acts on $X$ by compact operators. Hence for all $i\in \mathcal{I}$ there are $x_i \in X$ and $\phi_i \in X'$ such that

    \[\Delta(i) = \sum_{i}{x_{i}\otimes \phi_{i}}\]
    and
    \[i\cdot P_{0} = i\cdot \mathrm{id}_{T(X)} - \sum_{i}{T_{x_{i}} T_{\phi_{i}}} \in \mathcal{T}_{\mathcal{X}}\]
\end{remark}
\begin{proposition}
$J_{\mathcal{X},{\mathcal{I}}}$ is the ideal of compact operators of the right functional $R$-module \\
$T(\mathcal{X})\cdot \mathcal{I} = (T(X)\cdot \mathcal{I},\mathcal{I}\cdot T(X'),\widetilde g)$.
\end{proposition}
\begin{proof}
    We can represent elements of $\mathcal{L}_{R}(T(X))$ by (infinite) matrices by giving the action of the operator on each subspace $X^{\otimes n}$. In particular

    \[T_{x} = \begin{pmatrix}
    0 & 0 & 0 & 0 & \ldots\\
x & 0 & 0 & 0 & \ldots\\
0 & x & 0 & 0 & \ldots\\
0 & 0 & x & 0 & \ldots\\
\ldots
\end{pmatrix}, 
    T_{\phi} = \begin{pmatrix}
0 & \phi & 0 & 0 & \ldots\\
0 & 0 & \phi & 0 & \ldots\\
0 & 0 & 0 & \phi & \ldots\\
0 & 0 & 0 & 0 & \ldots\\
\ldots
\end{pmatrix},
    \pi_{0}\cdot{i} = \begin{pmatrix}
i & 0 & 0 & 0 & \ldots\\
0 & 0 & 0 & 0 & \ldots\\
0 & 0 & 0 & 0 & \ldots\\
0 & 0 & 0 & 0 & \ldots\\
\ldots
\end{pmatrix}\]
Here the symbol $x$ (resp. $\phi$) in the above matrix is to be understood as the operator $T_x$ (resp $T_\phi$) restricted to the corresponding subspace $X^{\otimes n}$. Moreover $\mathcal{K}_R(T(\mathcal{X}) \cdot \mathcal{I}) = T(X)\cdot \mathcal{I} \otimes_{R} \mathcal{I}\cdot  T(X') = T(X)\cdot \mathcal{I} \otimes_{R} T(X')$ because $\mathcal{I}$ has local units. Distributing the tensor product we get
\[\mathcal{K}_R(T(\mathcal{X} \cdot \mathcal{I})) = \bigoplus_{(n,m)\in \mathbb{N}^{2}}X^{\otimes n}\otimes \mathcal{I}\otimes (X')^{\otimes m},\] which is an ideal of $\mathcal{L}_{R}(T(X))$. We have $J_{\mathcal{X},{\mathcal{I}}}\subset \mathcal{K}_R(T(\mathcal{X}) \cdot \mathcal{I})$ because $P_{0} \in \mathcal{K}_R(T(\mathcal{X} \cdot \mathcal{I}))$ and any matrix of the form $T_{\phi}$ or $T_{x}$ multiplied by a finite matrix is still a finite matrix. In addition, for any $i\in \mathcal{I}$, $p = x_{1}\otimes \ldots \otimes x_{n}\in X^{\otimes n}$, $\psi = \phi_{1}\otimes \ldots \otimes \phi_{m} \in (X')^{\otimes m}$ :

\[\begin{pmatrix}
0 & \ldots & 0 & 0 & \ldots\\
\ldots & \ldots & 0 & 0 & \ldots\\
0 & ( p\cdot i )\otimes \psi  & 0 & 0 & \ldots\\
0 & 0 & 0 & 0 \\
\ldots
\end{pmatrix} = 
T_{x_{1}}\ldots T_{x_{n}}(i\cdot P_{0})T_{\phi_{1}}\ldots T_{\phi_{m}}\in \mathcal{T}_{\mathcal{X}}. \qedhere \]
\end{proof}
\begin{example} Let $\mathcal{X} = (M,M^*,\text{ev})$ with $M$ an $R$-$R$ bimodule that is finitely generated and projective on the right. We have 
    \[ \mathcal{O}_{\mathcal{X}} \cong \varinjlim_{k\in \mathbb{N}}\bigoplus_{d\in \mathbb{Z}} \text{Hom}_{-,R}(M^{\otimes k},M^{\otimes k+d})\] as rings. We take the convention that $M^{\otimes n}$ is the zero module for negative $n$. The colimit is taken over tensoring by the identity of $M$ and the product is given by the composition. Indeed, because $M_R$ is finitely generated and projective we have an isomorphism (see \cite{AAM87}) Hom$_{-,R}(M^{\otimes k},M^{\otimes l}) \cong M^{\otimes l} \otimes (M^{*})^{\otimes k}$. If we let $O$ be the ring obtained as the above direct sum of colimits then there is a ring homomorphism $\mathcal{T}_{\mathcal{X}} \rightarrow O$ obtained by restriction to the modules $M^{\otimes k}$ for sufficiently large $k$. The kernel of this map is clearly $J_\mathcal{X} = T(M) \otimes T(M^*)$. If $R$ is unital then $\mathcal{O}_\mathcal{X}$ coincides with the strong covariance ring of the bimodule $M$ (see \cite{Mey}).
    \end{example}
\begin{example}
    Let $\mathcal{X} = (R,R,\mu)$, $\mathcal{I} = R$. Let $R$ act on $R$ by multiplication. Then $T(X) = R^{(\mathbb{N})}$. Each $T_{x}$ acts on $T(X)$ by multiplication by $x$ in $R$ and by adding one to the degree, each $T_{\phi}$ multiplies by $\phi$ and subtracts one from the degree. Moreover $J_{\mathcal{X}} = \mathcal{K}_R(T(\mathcal{X})) = M_{\infty} (R)$ and \[
\mathcal{T}_{\mathcal{X}}
\cong {R\langle x,y\rangle}/{(xy-1)},
\qquad
\mathcal{O}_{\mathcal{X}}
\cong R[x,x^{-1}].
\]      
        More generally, if $\alpha$ is a ring automorphism of $R$ and each $r\in R$ acts on $R$ on the left by multiplication by $r$ and on the right by multiplication by $\alpha(r)$ then $\mathcal{O}_{\mathcal{X}} \cong R\rtimes_{\alpha}\mathbb{Z}$ is the crossed product ring of $R$ by $\mathbb{Z}$ with $\a$.
\end{example}
The Toeplitz ring has a universal characterization in terms of representations analogous to the one for $C^{\ast}$-algebras.

\begin{definition}{\cite[Definition 1.2]{CO11}}
    A covariant representation of $(R,\mathcal{X})$ is a quadruple $(S,T,\sigma,D)$ where
    \begin{enumerate}[label=(\roman*)]
        \item $D$ is a ring;
        \item $\sigma: R \to D$ is a ring homomorphism;
        \item $S:X' \rightarrow{} D$ and $T:X \rightarrow{} D$ are $R$-bimodule homomorphisms with respect to the bimodule structure induced by multiplication and by $\sigma$;
        \item $\forall \phi\in X', \forall x \in X, \sigma(g(\phi\otimes x)) = S(\phi)T(x)$.
    \end{enumerate}
\end{definition}
 If $j : R \to \mathcal{T}_{\mathcal{X}}$ is the natural inclusion, $T_{\mathcal{X}},S_{\mathcal{X}} : X,X' \rightarrow \mathcal{T}_{\mathcal{X}}$ map respectively $x$ to $T_{x}$ and $\phi$ to $T_{\phi}$ then $(S,T,j,\mathcal{T}_{\mathcal{X}})$ is a covariant representation. The ring homomorphism $j$ is injective by non-degeneracy. The ring $\mathcal{T}_\mathcal{X}$ is generated by $T_{\mathcal{X}}(X)$, $S_{\mathcal{X}}(X')$ and $j(R)$. The Toeplitz ring is universal in the following sense:
\begin{theorem}
Let $R$ be a ring with local units and let $\mathcal X$ be an $R$-correspondence.
If $(S,T,\sigma,D)$ is a covariant representation of $(R,\mathcal X)$, then there
exists a unique ring homomorphism
\[
\eta:\mathcal T_{\mathcal X}\to D
\]
such that $\eta\circ j=\sigma$, $\eta\circ T_{\mathcal X}=T$, and
$\eta\circ S_{\mathcal X}=S$.
\end{theorem}
\begin{proof}
See \cite[Theorem~1.7 and Proposition~4.2]{CO11}.
\end{proof}
The relative algebraic Cuntz-Pimsner ring $\mathcal{O}_{\mathcal{X},\mathcal{I}}$ also has a universal property \cite[Theorem 3.18]{CO11} but we will not use it in this article.
\section{Proofs of main results}
Let $\mathcal{X} = (X,X',g)$ be a correspondence over a ring with local units $R$, let $\mathcal{I}$ be a two-sided ideal of $R$ that has local units and acts on $X$ on the left by compact operators. Let $ E : \mathbf{Rings} \rightarrow \mathbf{Ab}$ be a homotopy invariant, split-exact, $M$-stable functor  (we refer to \cite{CMR07} and \cite{CT06} for the definitions of these notions and classical computations using quasihomomorphisms). There is a natural short exact sequence of rings

\[0 \longrightarrow{} \mathcal{K}_R(T(\mathcal{X} \cdot \mathcal{I})) \longrightarrow{} \mathcal{T}_{\mathcal{X}}\longrightarrow{} \mathcal{O}_{\mathcal{X},\mathcal{I}} \longrightarrow{} 0\]

Write $T(\mathcal{X} \cdot \mathcal{I}) = \mathcal{I} \oplus T^{1}(\mathcal{X} \cdot \mathcal{I})$ where $T^{1}(\mathcal{X} \cdot \mathcal{I}) = \bigoplus_{n \geq 1}{(\mathcal{X\cdot \mathcal{I}})^{\otimes n}}$. By Proposition 2.11 and $M$-stability of $E$ the upper left corner embedding $e : \mathcal{I} \rightarrow \mathcal{K}_R(T(\mathcal{X} \cdot \mathcal{I}))$ induces an isomorphism
\[E(\mathcal{K}_R(T(\mathcal{X} \cdot \mathcal{I}))) \cong E(\mathcal{I})\]

The main difficulty is thus to compute the value of $E$ on $\mathcal{T}_{\mathcal{X}}$. We have a map $j : R \rightarrow \mathcal{T}_{\mathcal{X}}$ by definition of the Toeplitz ring. We want to build a map that is the inverse to $j$ once we apply $E$. We define a quasi-homomorphism
\[
\pi =
\begin{tikzcd}
\mathcal{T}_\mathcal{X} \arrow[r, "\pi_{0}", shift left]
\arrow[r, "\pi_{1}"', shift right]
& {\mathcal{L}_{R}(T(\mathcal{X}))\rhd \mathcal{K}_{R}(T(\mathcal{X}))}.
\end{tikzcd}
\]
 by letting $\pi_{0}$ be the natural inclusion and by defining $\pi_{1}$ to be such that for all pure tensors $p \in T(X)$,
\newline
\[
\begin{aligned}
\pi_{1}(T_{x})(p)
&=
\begin{cases}
0 & \text{if } \deg(p)=0,\\
x \otimes p & \text{otherwise},
\end{cases}
\\[1em]
\pi_{1}(T_{\phi})(p)
&=
\begin{cases}
0 & \text{if } \deg(p)\le 1,\\
\phi(p_{1}) \otimes p_{2}\otimes \cdots \otimes p_{n} & \text{otherwise},
\end{cases}
\\[1em]
\pi_{1}(a)(p)
&=
\begin{cases}
0 & \text{if } \deg(p)=0,\\
a \cdot p & \text{otherwise}.
\end{cases}
\end{aligned}
\]

One can check that both $\pi_0$ and $\pi_1$ preserve the pairing and thus define a homomorphism on $\mathcal{T}_{\mathcal{X}}$ according to Theorem 2.22.
\begin{lemma}
    For all $\tau \in \mathcal{T}_{\mathcal{X}}$,
    \[\pi_{0}(\tau)-\pi_{1}(\tau) \in \mathcal{K}_{R}(T(X))\]
\end{lemma}
\begin{proof}
    Every element of the Toeplitz ring can be written as a finite sum of elements of the form $\tau = T_{p_{1}}\ldots T_{p_{k}}T_{\phi_{1}}\ldots T_{\phi_{l}}$ for $k,l \geq 0$, $p_1, \ldots p_k \in X$ and $\phi_1, \ldots, \phi_l \in X'$.  $\pi_{0}(\tau)-\pi_{1}(\tau)$ vanishes on homogeneous tensors of degree $\ge l+1$ and of degree $\le l-1$. Let $q = q_{1}\otimes\ldots \otimes q_{l}$ be a pure tensor of degree $l$. We have \[(\pi_{0}(\tau)-\pi_{1}(\tau))(q)=(p_{1}\otimes\ldots \otimes p_{k}\cdot (\phi_{1}(\ldots\phi_{l}(q_{1})\ldots)q_{l})\] Hence $\pi_{0}(\tau)-\pi_{1}(\tau)$ belongs to $\mathcal{K}_{R}(T(X)) = \bigoplus\limits_{n,m\in \mathbb{N}}X^{\otimes n}\otimes_R (X')^{\otimes m}$.
\end{proof}
\begin{theorem}
    Let $\mathcal{X}$ be an $R$-correspondence over a ring with local units $R$. Every homotopy invariant, split-exact, $M$-stable functor $E : \mathbf{Rings} \rightarrow \mathbf{Ab}$ sends the inclusion $j : R \rightarrow \mathcal{T}_{\mathcal{X}}$ to an isomorphism $E(\mathcal{T}_{\mathcal{X}}) \cong E(R)$.
\end{theorem}
\begin{proof}

    First we compute $E(\pi \circ j)$.
    Let $\iota = \iota_{R,T(X)} : R \rightarrow \mathcal{K}_{R}(T(X))$ be the inclusion in the upper left corner $(\iota(r)(p) = 0$ if deg$(p) \geq 1$ and $\iota(r)(p)= r \cdot p$ otherwise). The map $E(\iota)$ is an isomorphism by Proposition 2.11. We compute
    \[E(\pi \circ j) = E(j,\pi_{1}\circ j) = E (\iota + \pi_{1}\circ j, \pi_{1}\circ j) = E(\iota)\]
    as $\pi_{1} \circ j$ and $\iota$ are orthogonal quasi-homomorphisms (see \cite[Proposition 3.3]{CMR07}).
    \newline
    Now we compute $E(j) \circ E(\pi)$. Write $\mathcal{T}_\mathcal{X} = \mathcal{T}$ for simplicity. The map $j : R \rightarrow \mathcal{T}$ makes $(\mathcal{T},\mathcal{T},\mu)$ an $R$-$\mathcal{T}$ correspondence. Consider the following diagram :
    
   \begin{tikzcd}
\mathcal{T} \arrow[r, "\pi_{0}", shift left] \arrow[r, "\pi_{1}"', shift right] & {\mathcal{L}_{R}(T(\mathcal{X}))\rhd \mathcal{K}_{R}(T(\mathcal{X}))} \arrow[dd, "{[-\otimes \mathrm{id}_{\mathcal{T}}]}", shift left=2] \arrow[dd, "{[-\otimes \mathrm{id}_{\mathcal{T}}]}"', shift right=2] \\
                                                                                            &                                                                                                                                                                                                 \\
                                                                                            & {\mathcal{L}_{\mathcal{T}}(T(\mathcal{X})\otimes_{R} \mathcal{T})\rhd \mathcal{K}_{\mathcal{T}}(T(\mathcal{X})\otimes_{R} \mathcal{T})}                                                    
\end{tikzcd}

Let us first verify that $[-\otimes \mathrm{id}_{\mathcal{T}}]$ maps the compact operators to the compact operators.
$\mathcal{K}_{\mathcal{T}}(T(\mathcal{X})\otimes_{R} \mathcal{T}) = T(X)\otimes_{R} \mathcal{T}\otimes_{\mathcal{T}} \mathcal{T}\otimes_{R} T(X') \cong T(X)\otimes_{R} \mathcal{T}\otimes_{R} T(X')$
and 
$\mathcal{K}_{R}(T(\mathcal{X})) = T(X)\otimes T(X') \cong T(X)\otimes R \otimes T(X')$. Let $p \in X^{\otimes n},q \in X'^{\otimes m}$ and $r\in R$. The compact operator $p\otimes r \otimes q$ acting on $T(\mathcal{X})\otimes_{R} \mathcal{T}$ is equal to $p\otimes j(r) \otimes q \in T(X)\otimes_{R} \mathcal{T}\otimes_{R} T(X')$. Hence $[-\otimes \mathrm{id}_{\mathcal{T}}]$ restricted to $\mathcal{K} = \mathcal{K}_{R}(T(\mathcal{X}))$ is the map induced by $j$ on each coordinate. We write $[-\otimes \mathrm{id}_{\mathcal{T}}] : \mathcal{L}_{R}(T(\mathcal{X})) \rightarrow \mathcal{L}_{\mathcal{T}}(T(\mathcal{X})\otimes_{R} \mathcal{T})$ as $\tilde{j}$.
 $E(\tilde{j}|_\mathcal{K})$ equals $E(j)$ up to  stabilization isomorphisms. \[E(j) = E(\iota_{\mathcal{T},T(\mathcal{X})\otimes_{R} \mathcal{T}})^{-1}E(\tilde{j}|_\mathcal{K})E(\iota).\]
\newline
Here $\iota_{\mathcal{T},T(\mathcal{X})\otimes_{R} \mathcal{T}} : \mathcal{T} \rightarrow \mathcal{K}_{\mathcal{T}}(T(\mathcal{X})\otimes_{R} \mathcal{T})$ is the upper left corner embedding. We now define $\lambda_{0}, \lambda_{1} : \mathcal{T} \rightarrow \mathcal{L}_{\mathcal{T}}(T(\mathcal{X})\otimes_{R} \mathcal{T})$ two ring homomorphisms. We let $\lambda_{1}(\tau)$ be zero on $p\otimes \tau$ with deg$(p) \geq 1$ and the operator of left multiplication by $\tau$ on $R\otimes_R \mathcal{T} \cong \mathcal{T}$ (this isomorphism holds because we assume the left action to be non-degenerate). Let $\lambda_{0}(T_{x})$ send $p\otimes \tau$ to zero
if deg$(p)\geq 1$ and to $x\otimes p \otimes \tau$ otherwise. We let $\lambda_{0}(T_{\phi})$ send $p\otimes \tau$ to zero if deg$(p)\geq 2$ or if deg$(p) = 0$ and to $\phi(p)\cdot \tau$ if deg$(p) = 1$. Finally we let $\lambda_{0}(r)$ be zero on tensors of degree $\geq 1$ and send $\tau \in \mathcal{T}$ to $r\otimes \tau$. One easily checks that $\lambda_0$ preserves the pairing of $\mathcal{X}$. We can write elements of $\mathcal{L}_{\mathcal{T}}(T(\mathcal{X})\otimes_{R} \mathcal{T})$ as infinite matrices by giving their action on each subspace of the form $X^{\otimes n}\otimes \mathcal{T}$. We have 

    \[\lambda_0(T_{x}) = \begin{pmatrix}
    0 & 0 & 0 & \ldots\\
T_x & 0 & 0 & \ldots\\
0 & 0 & 0 & \ldots\\
\ldots
\end{pmatrix}, 
    \lambda_0(T_{\phi}) = \begin{pmatrix}
    0 & T_\phi & 0 & \ldots\\
0 & 0 & 0 & \ldots\\
0 & 0 & 0 & \ldots\\
\ldots
\end{pmatrix},
    \lambda_{0}(r) = \begin{pmatrix}
r & 0 & 0 & \ldots\\
0 & 0 & 0 & \ldots\\
0 & 0 & 0 & \ldots\\
\ldots
\end{pmatrix}\]

The last step of the proof is to build a polynomial homotopy
\[H : \pi_{0}\otimes \mathrm{id} \sim \lambda_{1} + \pi_{1}\otimes \mathrm{id}.\]
We define, using a “rotational” homotopy similar to that of \cite{CT06}
\begin{align*}
H(T_{x})
&= (1 - t^{2})\lambda_{0}(T_{x})
 + (2t - t^{3})\lambda_{1}(T_{x})
 + (\pi_{1} \otimes \mathrm{id})(T_{x}), \\
H(T_{\phi})
&= (1 - t^{2})\lambda_{0}(T_{\phi})
 + t\,\lambda_{1}(T_{\phi})
 + (\pi_{1} \otimes \mathrm{id})(T_{\phi}), \\
H(r)
&= r \cdot \mathrm{id}.
\end{align*}
In other words
\[H(T_{x}) =  \begin{pmatrix}
    (2t-t^3)T_x & 0 & 0 & 0 &\ldots\\
(1-t^2)T_x & 0 & 0 & 0 &\ldots\\
0 & T_x & 0 & 0 &\ldots\\
0 & 0 & T_x & 0 &\ldots\\
\ldots
\end{pmatrix}\]\\
\[H(T_{\phi}) = \begin{pmatrix}
    tT_\phi & (1-t^2)T_\phi & 0 & 0 &\ldots\\
0 & 0 & T_\phi & 0 &\ldots\\
0 & 0 & 0 & T_\phi &\ldots\\
0 & 0& 0 & 0 &\ldots\\
\ldots
\end{pmatrix} \]\\

A direct computation shows that $H$ indeed preserves the pairing
and hence defines a ring homomorphism 
\[H : \mathcal{T} \rightarrow \mathcal{L}_{\mathcal{T}}(T(\mathcal{X})\otimes_{R} \mathcal{T})[t]\] 
such that $H(0) = \pi_{0}\otimes \mathrm{id}$ and $H(1) = \lambda_{1} + \pi_{1}\otimes \mathrm{id}$. Hence, using the homotopy invariance of $E$ 
\[E(\tilde{j}\circ \pi)= E(\pi_{0}\otimes \mathrm{id},\pi_{1}\otimes \mathrm{id}) = E(\lambda_{1} +\pi_{1}\otimes \mathrm{id}, \pi_{1}\otimes \mathrm{id}) = E(\lambda_{1})\]
as $\lambda_{1}$ and $\pi_{1}\otimes \mathrm{id}$ are orthogonal homomorphisms. But $\lambda_{1} = \iota_{\mathcal{T},T(\mathcal{X})\otimes_{R} \mathcal{T}}$ is the upper left corner embedding of $\mathcal{T}$ in $\mathcal{K}_{\mathcal{T}}(T(\mathcal{X})\otimes_{R} \mathcal{T})$. Thus $E(\lambda_{1})$ is an isomorphism by Proposition 2.11 and $M$-stability of $E$.
\newline
Since $E(\tilde{j}\circ \pi) = E(\tilde{j}|_{\mathcal{K}})\circ E(\pi)$, this finishes the proof because we now have two isomorphisms which are inverses of each other :
\[E(j)E(\iota)^{-1}E(\pi) = \mathrm{id}_{\mathcal{T}},\]
\[E(\pi)E(j)E(\iota)^{-1} = \mathrm{id}_R. \qedhere\]
\end{proof}
\begin{lemma}
    Let $E$ be a split-exact and $M$-stable functor. Let $\iota_{R, T(\mathcal{X})}$ be the upper left corner embedding $R \rightarrow \mathcal{K}_{R}(T(\mathcal{X}))$.  We write $i : \mathcal{I} \rightarrow R$, and $i' : \mathcal{K}_R(T(\mathcal{X})\cdot \mathcal{I}) \rightarrow \mathcal{T}_{\mathcal{X}}$ for the inclusion maps. We have
    \[E(\pi)\circ E(i' \circ \iota_{\mathcal{I},T(\mathcal{X})\cdot\mathcal{I}}) = E(\iota_{R, T(\mathcal{X})})\circ(E(i)-E(\mathcal{X}))\]
    where $E(\mathcal{X})$ is defined to be the map induced by the structure of $\mathcal{I}$-$R$ correspondence on $R$.
\end{lemma}
\begin{proof}
    Recall that $ E(\mathcal{X}) = E(\iota_{R, R\oplus\mathcal{X}})^{-1}\circ E(\iota_{\mathcal{X}, R\oplus\mathcal{X}}) \circ  E(\Delta)$ where $\Delta : \mathcal{I} \rightarrow \mathcal{K}_{R}(\mathcal{X})$ is the left action, $\iota_{\mathcal{X}, R\oplus\mathcal{X}} : \mathcal{K}_{R}(\mathcal{X}) \rightarrow \mathcal{K}_{R}(R\oplus \mathcal{X})$ is the lower right corner embedding and $\iota_{R, R\oplus\mathcal{X}} : R \rightarrow \mathcal{K}_{R}(R\oplus\mathcal{X})$ is the upper left corner embedding. There is another natural inclusion $\iota_{\mathcal{X}, T(\mathcal{X})} : \mathcal{K}_{R}(\mathcal{X}) \rightarrow \mathcal{K}_{R}(T(\mathcal{X}))$.
    \[E(\pi)\circ E(i' \circ \iota_{\mathcal{I},T(\mathcal{X})\cdot\mathcal{I}}) = E(i'\circ \iota_{\mathcal{I},T(\mathcal{X})\cdot\mathcal{I}}, \pi_{1}\circ i'\circ \iota_{\mathcal{I},T(\mathcal{X})\cdot\mathcal{I}}) \]
For all $a \in \mathcal{I}$, write $\Delta(a) = \sum_{k} x_{k}\otimes \phi_{k}$ with $x_k \in X$ and $\phi_k \in X'$. We have
\[i'(\iota_{\mathcal{I},T(\mathcal{X})\cdot\mathcal{I}}(a)) = a\cdot p_0 = a\cdot \mathrm{id}_{T(\mathcal{X})} - \sum_{k} T_{x_{k}} T_{\phi_{k}} \]
By using the definition of $\pi_1$ we get  
$\pi_{1}(i'(\iota_{\mathcal{I},T(\mathcal{X})\cdot\mathcal{I}}(a))) = \iota_{\mathcal{X}, T(\mathcal{X})}(\Delta(a))$. The ring homomorphisms $\pi_0 \circ i'\circ\iota_{\mathcal{I},T(\mathcal{X})\cdot\mathcal{I}}$ and $\pi_1 \circ i'\circ\iota_{\mathcal{I},T(\mathcal{X})\cdot\mathcal{I}} : \mathcal{I} \rightarrow \mathcal{L}_R(T(\mathcal{X}))$ map into $\mathcal{K}_R(T(\mathcal{X}))$. \cite[Proposition 3.3]{CMR07} gives
\[E(\pi)\circ E(i' \circ\iota_{\mathcal{I},T(\mathcal{X})\cdot\mathcal{I}}) = E(\iota_{R,T(\mathcal{X})}\circ i)-E(\iota_{\mathcal{X},T(\mathcal{X})})\circ E(\Delta).
\]
It remains to show that $E(\iota_{R,T(\mathcal{X})})\circ E(\iota_{R, R\oplus\mathcal{X}})^{-1}\circ E(\iota_{\mathcal{X}, R\oplus\mathcal{X}}) = E(\iota_{\mathcal{X},T(\mathcal{X})})$. This is a direct consequence of the commutativity of the following diagram \[ \begin{tikzcd}
{\mathcal{K}_{R}(\mathcal{X})} \arrow[rr, "{\iota_{\mathcal{X}, R\oplus\mathcal{X}}}", hook] \arrow[rrdd, "{\iota_{\mathcal{X}, T(\mathcal{X})}}"', hook] &  & {\mathcal{K}_{R}(R\oplus\mathcal{X})} \arrow[dd, "{\iota_{R\oplus\mathcal{X}, T(\mathcal{X})}}"', hook] &  & R \arrow[ll, "{\iota_{R, R\oplus\mathcal{X}}}"', hook] \arrow[lldd, "{\iota_{R, T(\mathcal{X})}}", hook] \\
                                                                                                                                                            &  &                                                                                                           &  &                                                                                                          \\
                                                                                                                                                            &  & {\mathcal{K}_{R}(T(\mathcal{X}))}                                                                       &  &                                                                                                         
\end{tikzcd}\]
\end{proof}
\begin{theorem}
    Let $(E_n)_{n \in \mathbb{Z}} : \mathbf{Rings} \rightarrow \mathbf{Ab}$ be a sequence of homotopy invariant and $M$-stable functors satisfying excision. Let $R$ be a ring with local units. Let $\mathcal{X}$ be an $R$-correspondence and $\mathcal{I} \lhd R$ a two-sided ideal that has local units and acts on $\mathcal{X}$ on the left by compact operators. There is a long exact sequence
\[\begin{tikzcd}
\ldots \arrow[r] & E_{n}(\mathcal{I}) \arrow[rr, "E_n(i) - E_n(\mathcal{X})"] &  & E_{n}(R) \arrow[r, "E_n(j)"] & {{ E_{n}(\mathcal{O}_{\mathcal{X},\mathcal{I}}})} \arrow[r] & E_{n-1}(\mathcal{I}) \arrow[r] & \ldots
\end{tikzcd}\]
Here $j : R \rightarrow \mathcal{O}_{\mathcal{X}, \mathcal{I}}$ and $i : \mathcal{I} \rightarrow R$ are the natural inclusions.
\end{theorem}
We now show how some long exact sequences of KH and HP groups of some specific classes of rings can be obtained directly when realizing these specific rings as algebraic Cuntz-Pimsner rings. We fix $(E_n)_{n \in \mathbb{Z}} : \mathbf{Rings} \rightarrow \mathbf{Ab}$ a sequence of homotopy invariant, stable functors satisfying excision. As a direct corollary of Theorem B we get (see \cite[Example~5.5]{CO11} for details)
\begin{theorem}{(Pimsner-Voiculescu)}\\
    Let $R$ be a ring with local units, let $\alpha : R \rightarrow R$ be a ring automorphism. Let $R\rtimes_{\alpha}\mathbb{Z}$ be the crossed product ring of $R$ by $\alpha$. There is a long exact sequence : \[\begin{tikzcd}
\ldots \arrow[r] &  E_{n}(R) \arrow[r, "{1 - E_n(\alpha)}"] &  E_{n}(R) \arrow[r, "E_n(j)"] & { E_{n}(R\rtimes_{\alpha}\mathbb{Z})} \arrow[r] & E_{n-1}(R) \arrow[r] & \ldots
\end{tikzcd}\]
\end{theorem}
Let $Q=(Q^{0},Q^{1},r,s)$ be a quiver. Recall that a vertex $v\in Q^{0}$ is
regular if $0<|s^{-1}(v)|<\infty$, and let $\rho(Q)\subseteq Q^{0}$ denote
the set of regular vertices. It is known \cite[Example~5.8]{CO11} that the
Leavitt path algebra $L_{k}(Q)$ is the algebraic Cuntz-Pimsner ring associated
to the ring $R$, the ideal $\mathcal I$, and the $R$-correspondence
$\mathcal X=(X,X',\langle\cdot,\cdot\rangle)$ defined by
\[
R=\bigoplus_{v\in Q^{0}} k\cdot 1_{v},\qquad
\mathcal I=\bigoplus_{v\in\rho(Q)} k\cdot 1_{v},\qquad
X=\bigoplus_{e\in Q^{1}} k\cdot 1_{e},\qquad
X'=\bigoplus_{e\in Q^{1}} k\cdot 1_{e^{*}}.
\]
The evaluation map is given, for $e,f\in Q^{1}$, by
\[
\langle 1_{e^{*}},1_{f}\rangle=\delta_{e,f}\,1_{r(e)}
\]
and the bimodule structures are determined on generators by
\begin{alignat*}{2}
1_{e}\cdot 1_{v} &= \delta_{r(e),v}\,1_{e}, \qquad
&1_{v}\cdot 1_{e} &= \delta_{s(e),v}\,1_{e}, \\
1_{e^{*}}\cdot 1_{v} &= \delta_{s(e),v}\,1_{e^{*}}, \qquad
&1_{v}\cdot 1_{e^{*}} &= \delta_{r(e),v}\,1_{e^{*}}.
\end{alignat*}
The functional homomorphism of the correspondence is given by the maps $U :X \rightarrow R^{(Q^1)}$ and $V : X' \rightarrow R^{(Q^1)}$  defined by $U(1_e) = (\delta_{e,f}1_{r(e)})_{f\in Q^1}$ and $V(1_{e^*}) =(\delta_{e,f}1_{r(e)})_{f\in Q^1}$ for all $e\in Q^1$.

Let $N_{Q}'=(n_{x,y})_{x,y\in Q^{0}}$ be the adjacency matrix of $Q$, where
$n_{x,y}$ is the number of arrows from $x$ to $y$. Let $N_{Q}$ be the matrix
obtained from $N_{Q}'$ by removing the columns indexed by elements of $Q^{0}\setminus\rho(Q)$.

\begin{theorem}
    Let $k$ be a ring with local units. There is a long exact sequence $(n\in \mathbb{Z})$ \[\begin{tikzcd}
\ldots \arrow[r] & E_{n}(k)^{(\rho(Q))} \arrow[rr, "E_n(i) - N_Q"] &  & E_{n}(k)^{(Q^{0})} \arrow[r, "E_n(j)"] & E_n(L_k(Q)) \arrow[r] & \ldots
\end{tikzcd}\]
\end{theorem}
\begin{proof}

We just have to prove that the map \[E_n(\mathcal{X}) : E_n(\mathcal{I}) = E_{n}(k)^{(\rho(Q))} \rightarrow E_n(R) = E_n(k)^{(Q^0)}\] equals $N_Q$. $\mathcal{K}_R(\mathcal{X})$ is the set of $k$-linear combinations of elements of the form $1_e \otimes 1_{f^*}$ for $e,f \in Q^1$. The functional homomorphism of $\mathcal{X}$ described above induces the ring homomorphism $\rho : \mathcal{K}_R(\mathcal{X}) \rightarrow M_{Q^1}(R)$ which sends every element of the form $1_e\otimes 1_{f^*}$ to the matrix $(\delta_{e,e'}1_{r(e)}\otimes \delta_{f,f'}1_{r(f)})_{e',f'\in Q^1} = \delta_{e,f}1_{r(e)}$.
By Proposition 2.13, $E_n(\mathcal{X})$ equals the following map \[\begin{tikzcd}
E_n(\mathcal{I}) \arrow[r, "E_n(\Delta)"] & E_n(\mathcal {K}_R(\mathcal{X})) \arrow[r, "E_n(\rho)"] & E_n(M_{Q^1}(R)) \arrow[r, "\sim"] & E_n(R)
\end{tikzcd}\]
Here $\Delta : \mathcal{I} \rightarrow \mathcal {K}_R(\mathcal{X})$ is the left action. 
For all $v\in \rho(Q)$,
\[
\rho(\Delta(1_v))
= \rho\Big(\displaystyle\sum_{e\in s^{-1}(v)} 1_e \otimes 1_{e^*}\Big)
= \displaystyle\sum_{e\in s^{-1}(v)} 1_{r(e)}.
\qedhere \]\end{proof}

Nekrashevych algebras were introduced originally for $C^*$-algebras as Cuntz-Pimsner algebras naturally associated to self-similar groups \cite{Nek04}. Analogues in the discrete setting were recently found \cite{SS23}. Let $X$ be a finite set with $|X|\ge 2$, and let $X^{*}$ be the free monoid generated by $X$.
A self-similar group over $X$ is a group $G$ together with a faithful action on $X^{*}$ by length-preserving permutations such that, for every $g\in G$ and every $x\in X$, there exists an element $g|_{x}\in G$ satisfying
\[
g(xw)=g(x)\,g|_{x}(w)
\qquad\text{for all } w\in X^{*}.
\]

Let $k$ be a field, let $R = kG$ be the convolution algebra of $G$ with coefficient in $k$, let \[R = \bigoplus_{g\in G} k\cdot g = k^{(G)}, \qquad X = R^{(X)} = k^{(G\times X)}, \qquad X' = R^{(X)}\]
We define the $k$-bilinear scalar product 
\[
\Bigl\langle
\sum_{x\in X} \lambda_x\, x \cdot g_x,\;
\sum_{x\in X} \mu_x\, x \cdot h_x
\Bigr\rangle
=
\sum_{x\in X} \lambda_x \mu_x\, g_{x}^{-1}h_x
\]
Here $X$ and $X'$ are respectively right and left free $R$-modules, but as a left $R$-module we define $X$ by the relation $g\cdot x = g(x) \cdot g|_{x}$ for all $g\in G, x \in X$ (we have $g|_{x} \in G$) using the self-similarity condition.

The Nekrashevych algebra \cite{SS23} associated to the self-similar group $G$ (and the set $X$) with coefficients in $k$ is easily seen to be isomorphic to the algebraic Cuntz-Pimsner ring associated to the ring $R$ and the correspondence $\mathcal{X} = (X,X',\langle,\rangle)$ \[\mathcal{N}_k(G,X) = \mathcal{O}_{\mathcal{X}}.\]
As $X$ is finite, the left action obviously maps into the compact operators. The ring of compact operators is just $M_d(kG)$ where $d=|X|$.

\begin{theorem}
    For any self-similar group $(G,X)$ there is a long exact sequence $(n\in \mathbb{Z})$ \[\begin{tikzcd}
\ldots \arrow[r] & E_{n}(kG) \arrow[rr, "1 - E_n(\mathcal{X})"] &  & E_{n}(kG) \arrow[r, "E_n(j)"] & {E_n(\mathcal{N}_k(G,X))} \arrow[r] & \ldots
\end{tikzcd}\]
\end{theorem}

\section*{Acknowledgements}
The author would like to thank Professor Ralf Meyer for insightful discussions and valuable comments on earlier versions of this paper.


\begin{thebibliography}{CORIMB}
\bibitem[Abr83]{Abr83}
G.~D.~Abrams,
\textit{Morita equivalence for rings with local units},
Comm. Algebra \textbf{11} (1983), 801--837.

\bibitem[A{\'A}M87]{AAM87}
P.~N.~\'Anh and L.~M\'arki,
\textit{Morita equivalence for rings without identity},
Tsukuba J. Math. \textbf{11} (1987), no.~1, 1--16.

\bibitem[ABC09]{ABC09}
P.~Ara, M.~Brustenga, and G.~Corti{\~n}as,
\textit{K-theory of Leavitt path algebras},
M\"unster J. Math. \textbf{2} (2009), 5--34.

\bibitem[Bur25]{Bur25}
B.~Burgstaller,
\newblock Corner embeddings into algebras of compact operators in operator K-theory,
\newblock \href{https://arxiv.org/abs/2501.11504}{arXiv:2501.11504}, 2025.

\bibitem[CO11]{CO11}
T.~M.~Carlsen and E.~Ortega,
\textit{Algebraic Cuntz--Pimsner rings},
Proc. London Math. Soc. \textbf{103} (2011), no.~4, 601--653.

\bibitem[CT06]{CT06}
G.~Corti{\~n}as and A.~Thom,
\textit{Bivariant algebraic K-theory},
arXiv:math/0603531, 2006.

\bibitem[C97]{C97}
J.~Cuntz,
\textit{Bivariante K-Theorie f\"ur lokalkonvexe Algebren und der Chern-Connes-Charakter},
Doc. Math. \textbf{2} (1997), 139--182.

\bibitem[CMR07]{CMR07}
J.~Cuntz, R.~Meyer, and J.~M.~Rosenberg,
\textit{Topological and bivariant K-theory},
Springer, 2007.

\bibitem[CQ97]{CQ97}
J.~Cuntz and D.~Quillen,
\textit{Excision in bivariant periodic cyclic cohomology},
Invent. Math. \textbf{127} (1997), no.~1, 67--98.

\bibitem[Mey]{Mey}
R.~Meyer.
\newblock Bicategories in Noncommutative Geometry.
\href{https://www.uni-math.gwdg.de/rameyer/website/Bicategories/Chapter3.pdf}{Author's webpage}.
\bibitem[Nek04]{Nek04}
V.~Nekrashevych,
\textit{Cuntz--Pimsner algebras of group actions},
J. Operator Theory \textbf{51} (2004), no.~2, 223--249.

\bibitem[Pim97]{Pim97}
M.~V.~Pimsner,
\textit{A class of C$^\ast$-algebras generalizing both Cuntz--Krieger algebras and crossed products by $\mathbb{Z}$},
in \textit{Free Probability Theory} (Waterloo, ON, 1995),
Fields Inst. Commun. \textbf{12}, Amer. Math. Soc., 1997, pp.~189--212.

\bibitem[Rie82]{Rie82}
M.~A.~Rieffel,
\textit{Morita equivalence for operator algebras},
in \textit{Proc. Sympos. Pure Math.} \textbf{38}, 1982, pp.~285--298.

\bibitem[SS23]{SS23}
B.~Steinberg and N.~Szak\'acs,
\textit{On the simplicity of Nekrashevych algebras of contracting self-similar groups},
Math. Ann. \textbf{386} (2023), no.~3, 1391--1428.

\bibitem[Wei89]{Wei89}
C.~A.~Weibel,
\textit{Homotopy algebraic K-theory},
in \textit{Algebraic K-theory and Algebraic Number Theory} (Honolulu, HI, 1987),
Contemp. Math. \textbf{83}, Amer. Math. Soc., 1989, pp.~461--488.
\end{thebibliography}
\end{document}